\documentclass[12pt]{amsart}
\usepackage{a4wide}
\usepackage[T1]{fontenc}
\usepackage{amssymb,amsmath,amsthm,latexsym}
\usepackage{mathrsfs}
\usepackage[usenames,dvipsnames]{color}
\usepackage{euscript}
\usepackage{graphicx}
\usepackage{mdwlist}
\usepackage{mathtools,dsfont,wasysym}
\usepackage{stmaryrd}
\usepackage[dvipsnames]{xcolor}

\DeclareFontFamily{U}{mathx}{\hyphenchar\font45}
\DeclareFontShape{U}{mathx}{m}{n}{
	<5> <6> <7> <8> <9> <10>
	<10.95> <12> <14.4> <17.28> <20.74> <24.88>
	mathx10
}{}

\usepackage{todonotes}

\newtheorem{theorem}{Theorem}[section]
\newtheorem*{theoremA}{Theorem A}
\newtheorem*{theoremB}{Theorem B}
\newtheorem*{theoremC}{Theorem C}
\newtheorem{lemma}[theorem]{Lemma}
\newtheorem{corollary}[theorem]{Corollary}

\newtheorem{fact}[theorem]{Fact}
\newtheorem{claim}[theorem]{Claim}
\theoremstyle{remark}
\newtheorem{remark}[theorem]{Remark}
\theoremstyle{definition}

\numberwithin{equation}{section}
\makeatother

\newcommand{\nn}[1]{{\left\vert\kern-0.25ex\left\vert\kern-0.25ex\left\vert #1 
		\right\vert\kern-0.25ex\right\vert\kern-0.25ex\right\vert}}

\newcommand{\flag}{\mathord{\upharpoonright}}
\renewcommand{\leq}{\leqslant}
\renewcommand{\geq}{\geqslant}

\newcounter{smallromans}

\newenvironment{romanenumerate}
{\begin{list}{{\normalfont\textrm{(\roman{smallromans})}}}%
		{\usecounter{smallromans}\setlength{\itemindent}{0cm}%
			\setlength{\leftmargin}{5.5ex}\setlength{\labelwidth}{5.5ex}%
			\setlength{\topsep}{.5ex}\setlength{\partopsep}{.5ex}%
			\setlength{\itemsep}{0.1ex}}}%
	{\end{list}}

\newcommand{\R}{\mathbb{R}}
\newcommand{\N}{\mathbb{N}}

\newcommand{\K}{\mathbb{K}}
\newcommand{\e}{\varepsilon}
\newcommand{\p}{\varphi}
\newcommand{\n}{\left\Vert\cdot\right\Vert}

\DeclareMathOperator{\Span}{span}
\DeclareMathOperator{\supp}{supp}

\newcounter{smallromansdash}

	{\end{list}}

\newcounter{bigromans} 
	{\end{list}}

\begin{document}
	\title[Smooth norms in dense subspaces of Banach spaces]{Smooth norms in dense subspaces of Banach spaces}
	
	\author[S.~Dantas]{Sheldon Dantas}
	\address[S.~Dantas]{Department of Mathematics\\Faculty of Electrical Engineering\\
		Czech Technical University in Prague\\Technick\'a 2, 166 27 Praha 6\\ Czech Republic}
	\email{gildashe@fel.cvut.cz}
	
	\author[P.~H\'ajek]{Petr H\'ajek}
	\address[P.~H\'ajek]{Department of Mathematics\\Faculty of Electrical Engineering\\Czech Technical University in Prague\\Technick\'a 2, 166 27 Praha 6\\ Czech Republic}
	\email{hajek@math.cas.cz}
	
	\author[T.~Russo]{Tommaso Russo}
	\address[T.~Russo]{Department of Mathematics\\Faculty of Electrical Engineering\\
		Czech Technical University in Prague\\Technick\'a 2, 166 27 Praha 6\\ Czech Republic}
	\email{russotom@fel.cvut.cz}
	
	\thanks{Research of the first and second-named authors was supported by OPVVV CAAS CZ.02.1.01/0.0/0.0/16$\_$019/000077. Research of the third-named author was supported by the project International Mobility of Researchers in CTU CZ.02.2.69/0.0/0.0/16$\_$027/0008465 and by Gruppo Nazionale per l'Analisi Matematica, la Probabilit\`{a} e le loro Applicazioni (GNAMPA) of Istituto Nazionale di Alta Matematica (INdAM), Italy}
	
\date{\today}
	
\begin{abstract} In the first part of our paper, we show that $\ell_\infty$ has a dense linear subspace which admits an equivalent real analytic norm.
As a corollary, every separable Banach space, as well as  $\ell_1(\mathfrak{c})$, also has a dense linear subspace which admits an analytic renorming. By contrast, no dense subspace of $c_0(\omega_1)$ admits an analytic norm. In the second part, we prove (solving in particular an open problem of Guirao, Montesinos, and Zizler in \cite{GMZ}) that every Banach space with a  long unconditional Schauder basis contains a dense subspace that admits a $C^{\infty}$-smooth norm. Finally, we prove that there is a proper dense subspace of $\ell_{\infty}^{c}(\omega_1)$ that admits no G\^ateaux smooth norm. (Here, $\ell_{\infty}^{c} (\omega_1)$ denotes the Banach space of real-valued, bounded, and countably supported functions on $\omega_1$.)
\end{abstract}
\maketitle

	\section{Introduction}

It is well-known that the existence of an equivalent smooth norm has profound structural consequences for a Banach space $X$. For example, if $X$ has a $C^1$-smooth
renorming (or just a bump function) then it is an Asplund space, so the dual of every separable subspace of $X$ is also separable \cite{F}. If $X$ has a $C^2$-smooth renorming then $X$ either
contains a copy of $c_0$ or it is superreflexive \cite{FWZ}. Finally, if $X$ has a $C^\infty$-smooth renorming then it contains a copy of $c_0$ or $\ell_p$, where $p$ is an even integer \cite{D}. 

The proofs of these results, as well as many other results concerning the best smoothness of concrete Banach spaces, depend at some point on the completeness of the spaces in question.
Nevertheless, it is quite surprising that for separable spaces the completeness condition is in some sense also necessary. More precisely, Vanderwerff \cite{V} proved that every normed space with a countable algebraic basis admits a $C^{1}$-smooth renorming. This result was pushed further to get a $C^{\infty}$-smooth renorming \cite{H}. A slight reformulation of these results can be stated in the following form. Given any separable Banach space $X$, there exists a dense linear subspace of $X$ admitting an equivalent $C^{\infty}$-smooth norm. This leads naturally to the following, at a first glance rather bold, general question. 
\begin{center} 
{\it Given a (non-separable) Banach space $X$, is there a dense linear subspace admitting a $C^{k}$-smooth norm, where $k\in\N\cup\{\infty, \omega\}$?}
\end{center}
In the special case of $X=\ell_1(\Gamma)$ and $k=1$,  this problem was posed in the recent monograph by A. Guirao, V. Montesinos, and V. Zizler (see \cite[Problem 149]{GMZ}),
which was the starting point of our research. We obtain some partial positive results to this general problem, which remains open for $k=\infty$ (in the case $k=\omega$, we give a counterexample).
In the first part, we improve Vanderwerff's results as follows. (The proofs of the results presented in the Introduction are postponed to subsequent sections, where the necessary terminology is also explained; in particular, we refer to Section \ref{Sec: Prel} for the necessary background.)
	
\begin{theoremA}\begin{romanenumerate}
    \item $\ell_{\infty}$  admits a dense subspace with an analytic renorming;
	\item Every separable Banach space admits a dense subspace with an analytic renorming;
	\item $\ell_1(\mathfrak{c})$ admits a dense subspace with an analytic renorming.
	\end{romanenumerate}
\end{theoremA}

	We notice that (iii) is particularly surprising since the corresponding result for $c_0(\omega_1)$ fails to hold (see Theorem \ref{c0}), even though $c_0(\Gamma)$-spaces have much better smoothness properties than $\ell_1(\Gamma)$'s. Moreover, item (ii) answers positively \cite[Remark 3.4]{DFH}, where it was asked whether, in normed spaces with countable algebraic basis, every equivalent norm can be approximated by analytic norms.
\smallskip

	We also prove the following result, concerning $C^\infty$-smoothness.
	
	\begin{theoremB}  Let $X$ be a Banach space with long unconditional Schauder basis and let $Y$ be the linear span of such basis. Then, $Y$ admits a $C^{\infty}$-smooth norm.
	\end{theoremB}
	
In case of $\ell_1(\Gamma)$, Theorem B solves \cite[Problem 149]{GMZ}. Interestingly, Michal Johanis \cite{Johanis} recently proved that, under rather general assumptions, a normed space with a $C^k$-smooth bump also admits $C^k$-smooth partitions of the unity (see, \emph{e.g.}, \cite[Corollary 6]{Johanis}). It then follows that the normed spaces $Y$ considered in Theorem B also admit $C^\infty$-smooth partitions of the unity.

\smallskip
Finally, in the last part of our note, we prove a variant (for dense subspaces) of a well-known Partington's theorem \cite{P} concerning the existence of isometric copies of $\ell_{\infty}^c(\omega_1)$ in every renorming of $\ell_{\infty}^c(\omega_1)$ (we are grateful to Gilles Godefroy for suggesting us to dig in Partington's argument). Here, by $\ell_{\infty}^c(\Gamma)$ we understand the Banach space of all bounded scalar-valued functions on $\Gamma$ that are non-zero on at most countably many points in $\Gamma$, endowed with the canonical sup-norm; moreover, we denote by $\ell_{\infty}^{c,F}(\Gamma)$ its dense subspace comprising those functions that are finitely valued.
	
	\begin{theoremC} Every renorming of the space $\ell_{\infty}^{c,F}(\omega_1)$ contains an isometric copy of $\ell_{\infty}^{c,F}(\omega_1)$.
	\end{theoremC}

An immediate consequence of the present Theorem is that $\ell_{\infty}^{c,F}(\omega_1)$ admits no G\^{a}teaux smooth norm, in sharp contrast with Theorem A(i).

	\section{Preliminaries}\label{Sec: Prel}

	Let us present all the necessary background in order to avoid the reader jumping into specific references very often. The spaces we are considering throughout the paper are {\it real} normed spaces. We are following the notation from Banach space theory taken mainly from the books \cite{AK, FHHMZ, HJ}.
	
	\vspace{0.2cm}
	
	Let $X, Y$ be normed linear spaces. We say that the norm $\n$ of  $X$ is {\it $C^k$-smooth} if its $k$th Fr\'echet derivative exists and is continuous at every point of $X\setminus\{0\}$. The norm is said to be {\it $C^{\infty}$-smooth} if this holds for every $k \in \N$. We denote by $\mathcal{P}(^n X; Y)$ the normed linear space of all $n$-homogeneous continuous polynomials from $X$ into $Y$. If $U \subset X$ is an open subset, then we say that a function $f: U \rightarrow Y$ is {\it analytic} if, for every $a \in U$, there exist $P_n \in \mathcal{P}(^n X; Y)$ ($n\in\N \cup \{0\}$) and $\delta > 0$ such that, for all $x \in U(a, \delta)$,
	\begin{equation*}
	f(x) = \sum_{n=0}^{\infty} P_n (x - a).
	\end{equation*}	
	We denote by $C^{\omega}(U; Y)$ the vector space of all analytic functions from $U$ into $Y$. If $X, Y$ are Banach spaces over the field $\K$ (which can be the set of the real or complex numbers), $P_n \in \mathcal{P}(^n X; Y)$ for $n \in \N \cup \{0\}$, and $U \subset X$ is the domain of convergence of the power series $\sum_{n=0}^{\infty} P_n$,	
	then $f(x) = \sum_{n=0}^{\infty} P_n(x)$ is analytic in $U$ (see, for example, \cite[Chapter 1, Theorem 168]{HJ}). We are using this fact in the next section in Theorem \ref{equiv-analytic-norm-main}.

	\vspace{0.2cm}
	
	Let us recall that a convex set $C$ is said to be a {\it convex body} if its interior is not empty.	It is well-known that every bounded symmetric convex body induces an equivalent norm on $X$ via its {\it Minkowski functional}, which is defined as 
	\begin{equation*}
	\mu_B(x) := \inf \{t > 0: x \in tB \} \ \ \ (x \in X).
	\end{equation*}
	Throughout the paper, we are using the following property without any explicit reference. If $B \subset C \subset (1 + \delta) B$ for some $\delta > 0$, then we have that
	\begin{equation*}
	\frac{1}{1 + \delta} \mu_B \leq \mu_C \leq \mu_B.
	\end{equation*}
	By a \emph{renorming} of a space, we understand replacing a given norm $\n$ on a normed space $X$ with an equivalent norm which satisfies some desired property. Let us also clarify what we mean by approximate a norm. Given a normed space $(X,\n)$ and $\e > 0$, we say that a new norm $\nn\cdot$ {\it $\e$-approximates} $\n$ if
	\begin{equation*}
	\nn\cdot\leq\n\leq(1+\e)\nn\cdot.
	\end{equation*}

	In our main results, we are using the following technical lemma. The statement  we are presenting here can be found essentially in \cite[Chapter 5, Lemma 23]{HJ}. It is worth mentioning that it concerns normed spaces and it does not require completeness.
	
	\begin{lemma}[Implicit function theorem for Minkowski functionals] \label{minkowlemma} Let $(X, \n)$ be a normed space and $D$ be a nonempty open convex symmetric subset of $X$. Let $f: D \rightarrow \R$ be even, convex, and continuous. Suppose that there is $a > f(0)$ such that the level set $B := \{f \leq a\}$ is bounded and closed in $X$. Assume further that there is an open set $O$ with $\{f = a\} \subset O$ such that $f$ is $C^k$-smooth on $O$, where $k \in \N \cup \{\infty, \omega\}$. Then, the Minkowski functional $\mu$ on $B$ is an equivalent $C^k$-smooth norm on $X$.	
	\end{lemma}
	
	Let us notice that, in the assumption of the previous result, the set $B$ is assumed to be {\it a closed subset of $X$} and not merely a closed subset of $D$. Actually, $B$ being a closed subset of $D$, which is just a consequence of the continuity of $f$, would be not enough to conclude the result. Indeed, consider the space $X:=\ell_\infty^2$ and let $D$ be the open unit ball of $X$, where $\ell_{\infty}^2$ is $(\R^2, \n_{\infty})$. Define $f$ to be $0$ on $D$, which is obviously even, convex, and continuous on $D$. The set $B = \{ f \leq 1 \}$ is then $D$ itself (which is bounded, but not closed on $X$) and $\{f=1\}$ is the empty set. So, we can take $O = \emptyset$, where $f$ is trivially $C^{\infty}$-smooth. On the other hand, the Minkowski functional of $B = D = B_X$ is the norm of $X$, which is not differentiable.

	\vspace{0.2cm}	
	
	Some background on set theory is also needed. We adopt von Neumann's definition of ordinal number as the set of its predecessors. We regard cardinal numbers as initial ordinal numbers; accordingly, we write $\omega$ for $\aleph_0$, $\omega_1$ for $\aleph_1$, etc. We denote by $\mathfrak{c}$ the cardinality of continuum. The cardinality of a set $A$ will be denoted $|A|$. For a cardinal $\kappa$, we denote by $\kappa^+$ the smallest cardinal number that is strictly greater than $\kappa$. Finally, given a cardinal number $\kappa$, we denote by ${\rm cf}\,\kappa$ the \emph{cofinality} of $\kappa$, that is, the smallest cardinal number $\lambda$ such that $\kappa$ can be written as a union of $\lambda$ many sets each of cardinality less than $\kappa$. We refer to \cite{J} or \cite{K} for more on set-theoretical background.

	\section{Analytic norms}\label{Sec: Analytic}
	
	In this section we show that some dense subspaces of some Banach spaces admit an analytic renorming that approximates the original norm of the space. As a consequence, we prove Theorem A.
	
    In our first result, we are dealing with the normed space $\ell_\infty^F$, which consists of all finitely-valued sequences in $\ell_{\infty}$. Note that it is a dense linear subspace of $\ell_{\infty}$.

	\begin{theorem} \label{equiv-analytic-norm-main} The space $\ell_{\infty}^F$ admits an analytic norm that approximates $\n_{\infty}$.
	\end{theorem}

	\begin{proof} In order to prove this result, we use the Implicit function theorem for Minkowski functionals as stated in Lemma \ref{minkowlemma}. Denote by $x=(x(i))_{i=1}^{\infty}$ an element of $\ell_{\infty}$ and consider the following set
		\begin{equation*}
		U := \left\{ x \in \ell_{\infty}\colon \|x\|_{\infty} < 2 \ \mbox{and} \ \exists j \in \N, q \in (0,1)  \ \mbox{such that} \ |x (i)| < q, \forall i > j \right\}.	
		\end{equation*}
Clearly, $U$ is open and convex.\smallskip

		Now, we define $\p\colon U \rightarrow [0, \infty)$ by
		\begin{equation*}
		\p (x) := \sum_{i=1}^{\infty} x(i)^{2i+p} \ \ \ (x \in U)	
		\end{equation*}
		where $p$ is an even integer to be fixed later. Note that if $x \in U$, then there is $j \in \N$ and $q \in (0, 1)$ such that $|x(i)| < q$ for all $i > j$. Thus,
		\begin{equation*}
		\sum_{i=1}^{\infty}	x(i)^{2i + p} = \sum_{i=1}^j x(i)^{2i+p} + \sum_{i > j} x(i)^{2i+p} < \sum_{i=1}^j x(i)^{2i+p} + \sum_{i > j} q^{2i+p} < \infty.
		\end{equation*}
		Therefore, $\p$ is well-defined and it is analytic on $U$ (see, \emph{e.g.}, \cite[Chapter 1, Theorem 168]{HJ}).\smallskip

		Let $(\e_n)_{n=1}^{\infty}$ be a sequence of positive real numbers such that $\e_1 < 1$, $\e_{n+1} < \e_n$ for each $n \in \N$, and $\lim_{n \rightarrow \infty} \e_n = 0$. Define $T\colon \ell_{\infty}^F \rightarrow \ell_{\infty}$ by $T(x) := ((1+\e_i)x(i))_i$  for every $x=(x(i))_i \in \ell_{\infty}^F$. It is clear that for every $x \in \ell_{\infty}^F$, we have 
		\begin{equation}\label{T isomorphism}
		\|x\|_{\infty} \leq \|T(x)\|_{\infty} \leq (1 + \e_1)\|x\|_{\infty}. 
		\end{equation}
		Then, $T$ is an isomorphism from $\ell_{\infty}^F$ onto its image $Z:= 	T(\ell_{\infty}^F)$. We will prove that $Z$ admits an analytic norm close to $\n_{\infty}$. To do so, define $\psi := \p\flag_{U\cap Z}$. Then, $\psi$ is even, convex, and analytic in $U\cap Z$.\smallskip

		We first claim that $B_Z \subseteq U$. Indeed, let $z \in B_Z$ and take $x \in \ell_{\infty}^F$ to be such that $T(x) = z$. Since the range $x(\N)$ of $x$ is finite, there is $n_0 \in \N$ such that $\{x(i): i = 1, \ldots, n_0 \} = x(\N)$. From this, it follows that
\begin{equation*}
|x(i)| \leq \frac{1}{1 + \e_i} \leq \frac{1}{1 + \e_{n_0}}, \ \forall  i \in \N.
\end{equation*}
Thus, for each $i>n_0$, we have 

\begin{equation*}
    |z(i)|=(1+\e_i)|x_i|\leq\frac{1+\e_i}{1+\e_{n_0}}\leq \frac{1+\e_{n_0+1}}{1+\e_{n_0}} <1.
\end{equation*}
This proves that $z \in U$, as desired.\smallskip

Next, with the shorthand notation $\{\psi\leq1\}:=\{z\in U\cap Z\colon \psi(z)\leq1\}$, we claim that 
		\begin{equation} \label{minkow}
		(1-\e_1)B_Z \subseteq \{\psi\leq1\} \subseteq B_Z.
		\end{equation}
		Indeed, if $z \in B$, then $z\in Z$ and $\psi(z) \leq 1$. So, $\sum_{i=1}^{\infty} z(i)^{2i+p} \leq 1$, which implies that $|z(i)| \leq 1$, for every $i\in\N$, and then $\|z\|_{\infty} \leq 1$. This shows that $\{\psi\leq1\} \subseteq B_Z$. On the other hand, if $z\in(1-\e_1) B_Z$, then evidently $z\in U\cap Z$. Moreover, $\|z\|_{\infty} \leq1-\e_1$, whence
		\begin{equation*}
		\psi(z) = \sum_{i=1}^{\infty} z(i)^{2i+p} \leq \sum_{i=1}^{\infty} (1 - \e_1)^{2i+p} \leq 1,	
		\end{equation*}
		if $p$ is large enough (notice that $p$ depends {\it only} on $\e_1$). So, $(1-\e_1)B_Z \subseteq B$ and, therefore, our second claim is also proved.\smallskip
		
		Our last claim is that $\{\psi\leq1\}$ is a closed subset of $Z$. Indeed, since $B_Z\subseteq U$, $\psi$ is in particular continuous on $B_Z$. Thus, by (\ref{minkow}), $\{\psi\leq1\}$ is a closed subset of $B_Z$, whence of $Z$. \smallskip

		To conclude, we can apply Lemma \ref{minkowlemma}, with $D=O:=U\cap Z$ (notice that this is an open subset of $Z$), to get that the Minkowski functional $\mu$ of $\{\psi\leq1\}$ is an equivalent analytic norm on $Z = T(\ell_{\infty}^F)$. Using again (\ref{minkow}), we have that it is close to $\n_{\infty}$. Finally, (\ref{T isomorphism}) allows us to pull back the analytic norm to $\ell_\infty^F$, as desired.
	\end{proof}
	
	By an analogous argument as in Theorem \ref{equiv-analytic-norm-main}, we may prove the following statement. 
	
	\begin{theorem}  Let $K$ be a totally disconnected and separable compact space and let $\mathcal{B}$ be a basis for the topology of $K$, consisting of clopen sets. Then the linear span of the characteristic functions of $\mathcal{B}$ is a dense subspace of $C(K)$ that admits an analytic norm.

	\end{theorem}

	The next result is an extension of \cite[Corollary 4]{H}. 
	
	
	\begin{theorem} \label{hamel} Every normed space with a countable Hamel basis admits an analytic norm that approximates the original norm of the space.
	\end{theorem}
	
	\begin{proof} Let $X$ be a normed space with a countable Hamel basis $(\widetilde{e}_n)_{n=1}^{\infty}$. We may suppose without loss of generality that $(\widetilde{e}_n)_{n=1}^{\infty} \subset S_X$. Let us define the following system. Set $e_1 := \widetilde{e}_1$. Using the Hahn-Banach theorem, pick $e_1^* \in X^*$ so that $\langle e_1^*,e_1\rangle = 1$ and $\|e_1^*\| = 1$. Pick now $e_2 \in S_X$ to be such that $\Span\{e_1, e_2\} = \Span \{ \widetilde{e}_1, \widetilde{e}_2\}$ and $\langle e_1^*,e_2\rangle = 0$. Again by the Hahn-Banach theorem, pick $e_2^* \in X^*$ to be such that $\langle e_2^*,e_2\rangle=1$ and $\|e_2^*\| = 1$. By induction, there is $\{e_n; e_n^*\}_{n=1}^\infty \subset S_X \times S_{X^*}$ so that, for all $k\in\N$, 
		\begin{equation} \label{system}
		e_k \in \bigcap_{i < k} \ker e_i^*, \ \ \langle e_k^*,e_k\rangle=1, \ \ \mbox{and} \ \ \ 
		\Span\{e_n: n \in \N\} = \Span \{\widetilde{e}_n: n \in \N \}.
		\end{equation}

		Let us prove now that if $y = \sum_{i=1}^N \alpha_i e_i$ with $\|y\| = 1$ is an element of the linear span of $\{e_n: n \in \N\}$, then $|\alpha_i| \leq 2^{i-1}$ for each $i = 1, \ldots, N$. Indeed, we use induction and (\ref{system}). Note first that
		\begin{equation*}
		1 = \|y\| \geq |\langle e_1^*,y\rangle| = |\alpha_1|.
		\end{equation*}
		Now suppose that $|\alpha_i| \leq 2^{i-1}$ for every $i = 1, \ldots, N-1$. Then, since 
		\begin{equation*}
		1 = \|y\| \geq \left\langle e_N^*, \sum_{i=1}^N \alpha_i e_i\right\rangle = |\alpha_1\langle e_N^*,e_1\rangle + \alpha_2\langle e_N^*,e_2\rangle + \ldots + \alpha_N|,
		\end{equation*}
		we can apply the reverse triangle inequality to get
		\begin{equation*}
		|\alpha_N| \leq 1 + |\alpha_1| + |\alpha_2| + \ldots + |\alpha_{N-1}| \leq 1 + 1 + 2 + \ldots + 2^{N-2} = 2^{N-1}.
		\end{equation*}

		Since $X$ is separable, we can consider $X$ as a subspace of $\ell_\infty$. Fix $\e \in (0,1)$; by the fact that $\ell_{\infty}^F$ is dense in $\ell_{\infty}$, we can find, for each $n \in \N$, an element $x_n \in \ell_{\infty}^F$ such that 
		\begin{equation*} 
		\|x_n - e_n\|_{\infty} < \frac{\e}{4^n}.
		\end{equation*}
		We will prove that there is an isomorphism from $\Span \{e_n: n \in \N\}$ onto $\Span \{x_n: n \in \N\}$ by using the mapping $T: e_i \mapsto x_i$. Indeed, let $x = \sum_{i=1}^N \alpha_i e_i$. So, $T(x) = \sum_{i=1}^N \alpha_i x_i$. If $\|x\| = 1$, we have that 
		\begin{eqnarray*}
			\left| \|T(x)\| - \|x\| \right| &=& \left| \left\| \sum_{i=1}^N \alpha_i x_i \right\| - \left\| \sum_{i=1}^N \alpha_i e_i \right\| \right| \\ 
			&\leq& \left\| \sum_{i=1}^N \alpha_i (x_i - e_i) \right\| \\
			&\leq& \sum_{i=1}^N 2^{i-1} \|x_i - e_i\|_{\infty} \\
			&\leq& \sum_{i=1}^{\infty} \frac{\e}{2^{i+1}} < \e= \e \cdot\|x\|.
		\end{eqnarray*}
Consequently, for every $x \in \Span \{e_n: n \in \N\}$, we have that
\begin{equation*}
(1 - \e) \|x\| \leq \|T(x)\| \leq (1 + \e)\|x\|.
\end{equation*}
Therefore, $T$ is an isomorphism as desired. By Theorem \ref{equiv-analytic-norm-main}, $\ell_{\infty}^F$ admits an analytic equivalent norm which is close to $\n_\infty$ and then so does $\Span \{e_n: n \in \N\} = X$.
	\end{proof}
	
	\begin{remark} Let us mention that a somewhat shorter proof of the above theorem is possible, by selecting a better system of coordinates in $X$. More precisely, the classical Marku\v{s}evi\v{c}'s theorem \cite{Markushevich} (see, \emph{e.g.}, \cite[Lemma 1.21]{HMVZ}) yields us an M-basis $\{e_n;e^*_n\}_{n=1}^\infty$ such that $X=\Span\{e_n\}_{n=1}^\infty$. It is then sufficient to find vectors $x_n\in\ell_\infty^F$ such that $$\|e_n-x_n\|_\infty \cdot\|e^*_n\|\leq \e\cdot2^{-n}$$
		and conclude as in the above proof. We decided to present the above argument, as it is self-contained and we do not really need the existence of an M-basis in the argument.
	\end{remark}
	
	We have the following immediate consequence of Theorem \ref{hamel}.
	\begin{corollary} Let $X$ be a separable Banach space. Then there is a dense subspace $Y$ of $X$ which admits an analytic norm that approximates the original norm of $Y$.
	\end{corollary}

	As we pointed out in the introduction of this paper, A. Guirao, V. Montesinos, and V. Zizler asked in their recent monograph \cite{GMZ} the following question: if $F$ is the normed space of all finitely supported vectors in $\ell_1(\Gamma)$, with $\Gamma$ uncountable, endowed with the $\ell_1$-norm, then does $F$ admit a Fr\'echet smooth norm? (see \cite[Problem 149]{GMZ}). In the next results, we are dealing with this problem and we are using Theorem \ref{equiv-analytic-norm-main} as a tool to prove them.
	
	\begin{theorem} \label{l1} The normed space $(F,\n_1)$ of all finitely supported vectors in $\ell_1(\mathfrak{c})$ admits an analytic norm that approximates $\n_1$.
	\end{theorem}
	
	\begin{proof} It is a standard fact that $\ell_1(\mathfrak{c})$ is isometric to a subspace of $\ell_{\infty}$ (see, for example, \cite[Exercise 3.143, p. 169]{FHHMZ}). Consider $(e_{\lambda})_{\lambda<\mathfrak{c}}$ the canonical basis of $\ell_1(\mathfrak{c})$. Pick $(x_{\lambda})_{\lambda<\mathfrak{c}} \subseteq \ell_{\infty}^F$ to be such that $\|x_{\lambda} - e_{\lambda}\|_{\infty} < \e$ for each $\lambda$, where $\e \in (0, 1)$.
	
    An argument as in the proof of Theorem \ref{hamel} shows that the linear spans of $(e_{\lambda})_{\lambda<\mathfrak{c}}$ and $(x_{\lambda})_{\lambda<\mathfrak{c}}$ are isomorphic via the mapping $(e_{\lambda})_{\lambda<\mathfrak{c}} \mapsto (x_{\lambda})_{\lambda<\mathfrak{c}}$ (also see, \emph{e.g.}, \cite[p. 331]{Jameson}). It then follows from Theorem \ref{equiv-analytic-norm-main} that $F=\Span \{e_\lambda\}_{\lambda<\mathfrak{c}}$ admits an analytic  norm.
	\end{proof}
	
	We have the following consequence of Theorem \ref{l1}, which proves Theorem A(iii).
	
	\begin{corollary}\label{Cor: analytic in dense of l1c} Let $S$ be a set with $|S|\leq\mathfrak{c}$. Then, every dense subspace of $\ell_1(S)$ contains a further dense subspace which admits an analytic norm.
	\end{corollary}
	
	\begin{proof} Let $(e_{\lambda})_{\lambda \in S}$ be the canonical basis of $\ell_1(S)$. It is easy to see, perturbing $(e_{\lambda})_{\lambda \in S}$ as in Theorem \ref{l1}, that every dense subspace of $\ell_1(S)$ contains a subspace isomorphic to $\Span\{e_{\lambda} \}_{\lambda \in S}$ (this well-known fact is recorded, \emph{e.g.}, in \cite[Theorem 2.3]{HR}). By using Theorem \ref{l1}, every dense subspace of $\ell_1(S)$ contains a further dense subspace which admits an analytic norm.
	\end{proof} 
	
In the next theorem we shall show that the corresponding result for $c_0(\omega_1)$ fails. Before we formulate the result, we shall record a couple of known results that we require in the proof. First, we need the following known lemma, a simple consequence of the fact that polynomials on $c_0(\Gamma)$ are weakly sequentially continuous. (Indeed, more generally, real-valued Fr\'echet differentiable functions on $c_0(\Gamma)$ with locally uniformly continuous derivatives are weakly sequentially continuous, \cite[Corollary 8]{H1}.) We present a proof for the sake of completeness; in the argument, we denote by $\mathcal{L}_s (^n X)$ the space of all symmetric continuous $n$-linear forms on the normed space $X$.

\begin{lemma} \label{symmetric} Let $\Gamma$ be any set and $M \in \mathcal{L}_s (^n c_0(\Gamma))$. Then, $M$ is countably supported. More precisely, there exists a countable set $A \subseteq \Gamma$ such that 
\begin{equation*}
M(e_{\gamma_1}, \ldots, e_{\gamma_n}) = 0 \ \mbox{for all} \ \{\gamma_1, \ldots, \gamma_n \} \not\subseteq A.	
\end{equation*}	
\end{lemma}

\begin{proof} The result is trivially true when $\Gamma$ is countable; thus, we assume $|\Gamma|\geq\omega_1$. By a result due to Pe\l{}czy\'nski and Bogdanowicz (see, \emph{e.g.}, \cite[Chapter 3, Corollary 59]{HJ}), the polynomial associated to $M$ is weakly sequentially continuous; thus, by the Polarisation Formula (see, for example, \cite[Chapter 1, Proposition 11]{HJ}), the same is true for $M$.

Suppose now that the conclusion of the result is false. Then, for each countable set $A \subseteq \Gamma$, there exists a finite set $\{ \gamma_1, \ldots, \gamma_n \} \not\subseteq A$ such that $M(e_{\gamma_1}, \ldots, e_{\gamma_n}) \not= 0$. Since $M$ is symmetric, by a slight abuse of notation, let us denote $M(e_{\gamma_1}, \ldots, e_{\gamma_n})$ by $M(F)$, where $F := \{\gamma_1, \ldots, \gamma_n\} \subseteq \Gamma$. In other words, for every countable set $A \subseteq \Gamma$, there exists $F \not\subseteq A$ with $|F| = n$ such that $M(F) \neq 0$. In particular, we may fix $F_0\subseteq\Gamma$ with $M(F_0) \neq 0$ and $|F_0|=n$. Then, our assumption allows us to choose $F_1 \subseteq \Gamma$ with $F_1 \not\subseteq F_0$ and $|F_1| = n$ such that $M(F_1) \not= 0$. Continuing by transfinite induction, we find $(F_{\alpha})_{\alpha < \omega_1}$ with $F_{\alpha} \not\subseteq \bigcup_{\xi < \alpha} F_{\xi}$ and $|F_{\alpha}| = n$ such that $M(F_{\alpha}) \not= 0$. Notice that $F_\alpha\neq F_\beta$ for distinct $\alpha,\beta<\omega_1$.\smallskip

By the uncountable cofinality of $\omega_1$, we may assume, without loss of generality, that
\begin{equation}\label{c0_11}
	|M(F_{\alpha})| \geq \e, \ \ \forall \alpha < \omega_1 
\end{equation}
for some $\e > 0$, where $F_{\alpha} := \{\gamma_1^{\alpha}, \ldots, \gamma_n^{\alpha} \}$. By the $\Delta$-system Lemma (see, for example, \cite[Lemma III.2.6]{K}), we may assume that there exists $\Delta$ such that $F_{\alpha} \cap F_{\beta} = \Delta$ for distinct $\alpha,\beta<\omega_1$. In other words, we can write
\begin{equation*}
F_{\alpha} = \Delta \cup \{ \gamma_{k+1}^{\alpha}, \ldots, \gamma_{n}^{\alpha} \},	
\end{equation*}
where $\Delta = \{ \gamma_1^{\alpha}, \ldots, \gamma_k^{\alpha} \}$ for some $k \leq n$ such that $\gamma_j^{\alpha} = \gamma_j^{\beta}$ for every $\alpha \not= \beta$ in $\Gamma$ with $\alpha, \beta < \omega_1$ and $j=1, \ldots, k$.

Let $(\alpha_j)_{j=1}^\infty$ be an injective sequence in $\omega_1$. Write $\gamma_j^0 := \gamma_j^{\alpha}$ for all $\alpha < \omega_1$ and $j = 1, 2, \ldots, k$. Notice that $\{ (e_{\gamma_{k+1}^{\alpha_j}}, \ldots, e_{\gamma_n^{\alpha_j}} )\}_j$ is weakly null in $c_0(\Gamma)^{n-k}$. Moreover, by (\ref{c0_11}), $|M(F_{\alpha_j})| \geq \e$. Reordering and using again that $M$ is symmetric, we have that
\begin{equation*}
\left| M \left(e_{\gamma_1^{\alpha_j}}, \ldots, e_{\gamma_k^{\alpha_j}}, e_{\gamma_{k+1}^{\alpha_j}}, \ldots, e_{\gamma_n^{\alpha_j}} \right)	\right| \geq \e.
\end{equation*}

On the other hand, by the property the set $\Delta$ satisfies and by the fact that $M$ is weakly sequentially continuous, we have that
\begin{equation*}
\left| M \left(e_{\gamma_1^{\alpha_j}}, \ldots, e_{\gamma_k^{\alpha_j}}, e_{\gamma_{k+1}^{\alpha_j}}, \ldots, e_{\gamma_n^{\alpha_j}} \right)	\right| \rightarrow \left| M \left(e_{\gamma_1^0}, \ldots, e_{\gamma_k^0}, 0, \ldots, 0 \right) \right| = 0,
\end{equation*}
which is a contradiction.
\end{proof} 
 
\begin{fact} Let $E$ be a normed space that admits an analytic norm $\n$. Then there exists an analytic function $f\colon E\to\R$ such that $f(0)= 0$ and $\inf f(S_E)>1$.
\end{fact}
\begin{proof} Let $H$ be an hyperplane on $E$ and $H_1$ be an affine hyperplane obtained by translation of $H$. The function $\tilde{f}$ obtained by restricting $\n$ to $H_1$ is evidently analytic, coercive (\emph{i.e.}, $\lim_{\|x\|\to\infty}\tilde{f}(x)=\infty$), and positive; therefore, we can find a coercive, positive, and analytic function defined on $H$. It is easy to extend such function to an analytic and coercive function on $E$; finally, the function $f$ is obtained by translation and scaling.
\end{proof}

\begin{theorem} \label{c0} No dense subspace of $c_0(\omega_1)$ admits an analytic norm.	
\end{theorem}

\begin{proof} Let $E$ be a dense subspace of $c_0(\omega_1)$ and suppose by contradiction that $E$ admits an analytic norm $\n$. Let $f\colon E\to\R$ be a function as in the previous fact. Since $f$ is analytic on $E$, there exists $\delta>0$ such that
\begin{equation*}
f(x) = \sum_{n=0}^{\infty} P_n(x),\, \qquad (x\in E, \|x\|<\delta),
\end{equation*}
where $P_n\in \mathcal{P}(^nE)$. Since each $P_n \in \mathcal{P}(^n E)$ is uniformly continuous on bounded sets, there exists a unique extension $\widetilde{P}_n \in \mathcal{P}(^n c_0(\omega_1))$ with $\|P_n\|_E = \|\widetilde{P}_n\|_{c_0(\omega_1)}$. 

Appeal to Lemma \ref{symmetric} yields that, for every $n\in\N\cup\{0\}$, there is a countable set $A_n \subseteq\Gamma$ such that for all $\{ \gamma_1, \ldots, \gamma_n \} \not\subseteq A_n$, we have that $\widetilde{M}_n (e_{\gamma_1}, \ldots, e_{\gamma_n}) = 0$. Again by the Polarisation Formula, we obtain that each $\widetilde{P}_n$ is {\it countably supported} in the sense that
\begin{equation*}
\widetilde{P}_n (x) = \widetilde{P}_n \left( x\flag_{A_n} \right) \qquad (x\in c_0(\omega_1)).
\end{equation*}
Now, by taking $A := \bigcup A_n$, we have that $A$ is still countable and $\widetilde{P}_n (x) = \widetilde{P}_n \left( x\flag_{A} \right)$ for every $n \in \N$ and every $x$ in $c_0(\omega_1)$. Therefore, 
\begin{equation*}
f(x) = f\left( x\flag_A \right)	\qquad (x\in E, \|x\|<\delta).
\end{equation*}
Now, since $E$ is dense in $c_0(\omega_1)$, we may choose $x \in E$ such that $\|x\|=1$ and $\|x\flag_A\|<\e$, where $\e>0$ is very small. By the continuity of $f$, the value $f\left(x\flag_A \right)$ must be small as well. On the other hand, we have that $f(x) \geq \inf f(S_E) > 1$, which gives us the desired contradiction.
\end{proof}

	\section{$C^\infty$-smoothness}
	This section is dedicated to the construction of smooth norms in presence of an unconditional basis; in particular, the main result of the section readily implies Theorem B. Before we proceed, we need to recall one definition (see, \emph{e.g.}, \cite[\S 7.3]{HMVZ}). A set $\{e_\gamma\}_{\gamma \in \Gamma}$ in a Banach space $X$ is called an \emph{unconditional Schauder basis} of $X$ if for every $x \in X$ there is a unique family of real numbers $\{a_{\gamma}\}_{\gamma \in \Gamma}$ such that $x = \sum_{\gamma \in \Gamma} a_{\gamma} x_{\gamma}$ in the following sense: for every $\e > 0$, there is a finite subset $F \subset \Gamma$ such that
	\begin{equation*}
	\left\| x - \sum_{\gamma \in G} a_{\gamma} e_{\gamma} \right\| < \e,
	\end{equation*}
	whenever $F\subseteq G$. If $\{ e_{\gamma} \}_{\gamma \in \Gamma}$ is an unconditional basis of $X$ and $A$ is a subset of $\Gamma$, then there is a naturally defined bounded linear projection $P_A$ from $X$ onto $\overline{\Span}\{e_\gamma\}_{\gamma\in A}$ defined by $P_A(x) = \sum_{\gamma \in A} \langle e_\gamma^*,x\rangle e_{\gamma}$. The number $\sup_{A\subseteq\Gamma} \|P_A\|$ is called the \emph{suppression constant} of the basis; accordingly, a basis is \emph{suppression $1$-unconditional} when $\sup_{A\subseteq \Gamma}\|P_A\|=1$.

	\begin{theorem} \label{unconditional} Let $X$ be a Banach space with a suppression $1$-unconditional Schauder basis $\{e_{\gamma}\}_{\gamma \in \Gamma}$ and set $Y:= \Span \{e_{\gamma}\}_{\gamma \in \Gamma}$. Then, $Y$ is a dense subspace of $X$ which admits a $C^{\infty}$-smooth norm that approximates the original one.	
	\end{theorem}
	
	\begin{proof} Let $\{e_\gamma\}_{\gamma\in\Gamma}$ be a suppression $1$-unconditional Schauder basis for the Banach space $(X,\n)$ and set $Y:=\Span\{e_\gamma\}_{\gamma\in\Gamma}$. We shall start by fixing some parameters. Let $(\e_k)_{k=0}^\infty$ be a strictly decreasing sequence of positive real numbers such that
		\begin{equation}\label{seq1}
		\e_k \searrow 0 \ \ \ \mbox{and} \ \ \ \frac{1 + \e_{k+1}}{1 + \e_{k}} \nearrow 1.
		\end{equation}
		Now, let $(\theta_k)_{k=0}^\infty$ be another decreasing sequence of positive real numbers such that
		\begin{equation} \label{seq2}
		\theta_k \searrow 0 \ \ \ \mbox{and} \ \ \ \frac{1 + \e_{k+1}}{1 + \e_{k}} < 1 - 2 \theta_{k+1} \ \mbox{for every} \ k\geq0.
		\end{equation}
		
		For each finite subset $A \subseteq\Gamma$, we have that $P_A(X)$ is a finite-dimensional Banach space and therefore, we can pick a $C^\infty$-smooth norm $\n_{(s),A}$ on $P_A(X)$ such that
		\begin{equation} \label{smooth-norm}
		\frac{1}{1 + \theta_{|A|}}\n \leq\n_{(s),A}\leq\n.
		\end{equation}
		
		We are now in position to define the following norm $\n_f$ on $Y$:
		\begin{equation*}
		\|x\|_f:=\sup_{|A|<\omega} (1 + \e_{|A|})\cdot\|P_A(x)\| \qquad(x\in Y).	
		\end{equation*}
		Let us notice that this new norm approximates the original one:
		\begin{equation} \label{f-normal-approx}
		\n\leq\n_f\leq (1+\e_1)\n.
		\end{equation}
		Moreover, for every $x\in Y$, we clearly have
		\begin{equation*}\|x\|_f = \max_{A\subseteq\supp x} (1 + \e_{|A|})\cdot\|P_A(x)\|.
		\end{equation*}
		
		We shall now prove that the above maximum is, indeed, locally attained in a strong sense. More precisely, we prove the following claim.
		\begin{claim}\label{Claim: jump-ngh} Let $x \in Y$ be such that $\|x\|_f \leq 1$. Set $n= |\supp x|$ and consider the open neighbourhood $\mathcal{O}_x$ of $x$ given by
			\begin{equation*} \mathcal{O}_x := \{y\in Y\colon\|y-x\|_f< \theta_{n+1}\}.
			\end{equation*}
			Then, for each $y \in \mathcal{O}_x$ and $A\nsubseteq\supp x$ we have
			\begin{equation}\label{jump-ngh}
			(1 + \e_{|A|})\cdot\|P_A(y)\|_{(s),A} \leq 1-\theta_{|A|}.
			\end{equation}
		\end{claim}
		
		\begin{proof}[Proof of Claim \ref{Claim: jump-ngh}] We start proving the following stronger estimate, valid for the vector $x$ and for every $A\nsubseteq\supp x$:
			\begin{equation} \label{jump}
			(1 + \e_{|A|})\|P_A(x)\|\leq 
			\begin{cases} \displaystyle
			\frac{1+\e_{|A|}}{1 + \e_{|A|-1}}, & \mbox{if}  \ |A| \leq n, \vspace{0.3cm} \\  \displaystyle
			\frac{1 + \e_{n+1}}{1 + \e_n}, & \mbox{if} \ |A| > n.
			\end{cases}
			\end{equation}
			Indeed, set $B := \supp x$ and take any $A\nsubseteq B$. Notice that we may assume $A \cap B \not= \emptyset$ since, otherwise, (\ref{jump}) is trivially true. Since $\|x\|_f\leq 1$, we have
			\begin{equation*}
			(1+\e_{|A|})\|P_A(x)\| = \frac{1+\e_{|A|}}{1+\e_{|A\cap B|}} (1+\e_{|A\cap B|})\|P_{A \cap B}(x)\|\leq
			\end{equation*}
			\begin{equation*}
			\leq \frac{1+\e_{|A|}}{1+\e_{|A \cap B|}}\|x\|_f \leq \frac{1+\e_{|A|}}{1+\e_{|A \cap B|}}.    
			\end{equation*}
			
			Assume first that $|A| \leq n$. Since $A\nsubseteq B$, we have $|A\cap B|\leq|A|-1$; thus, since $(\e_k)_{k=0}^\infty$ is decreasing,
			\begin{equation*}
			(1+\e_{|A|})\|P_A(x)\| \leq\frac{1+\e_{|A|}} {1+\e_{|A\cap B|}} \leq\frac{1+\e_{|A|}}{1+\e_{|A|-1}}, 	
			\end{equation*}
			which proves the first part of (\ref{jump}). On the other hand, if $|A| > n$, then, using again that $(\e_k)_{k=0}^{\infty}$ is decreasing and that $|B| = n$, we have
			\begin{equation*}
			(1+\e_{|A|})\|P_A(x)\|\leq\frac{1+\e_{|A|}}{1+ \e_{|A\cap B|}}\leq\frac{1+\e_{|A|}}{1+\e_{|B|}}\leq \frac{1+\e_{n+1}}{1+\e_{|B|}}=\frac{1+\e_{n+1}}{1+\e_n}.	\end{equation*}
			Therefore, (\ref{jump}) is proved.\smallskip
			
			We now pass to the proof of (\ref{jump-ngh}). Given any $y\in\mathcal{O}_x$ and any finite set $A$ with $A\nsubseteq \supp x$, we have:
			\begin{equation}\label{eq: from jump to ngbh}
			\begin{split}
			 (1+\e_{|A|})\|P_A(y)\|_{(s),A}&\stackrel{(\ref{smooth-norm})} {\leq} (1+\e_{|A|})\|P_A(y)\| \\& \leq (1+\e_{|A|})\|P_A(x)\| + (1+\e_{|A|})\|P_A(y-x)\| \\&
			\leq(1+\e_{|A|})\|P_A(x)\|+\|y-x\|_f < (1+\e_{|A|})\|P_A(x)\| + \theta_{n+1}.
			\end{split}
			\end{equation}
			
			If $|A|\leq n$, we continue from (\ref{eq: from jump to ngbh}) and we obtain
			$$(1+\e_{|A|})\|P_A(y)\|_{(s),A}< (1+\e_{|A|})\|P_A(x)\|+ \theta_{n+1}\stackrel{(\ref{jump})}{\leq} \frac{1+\e_{|A|}} {1+\e_{|A|-1}}+\theta_{|A|}\stackrel{(\ref{seq2})}{\leq} 1-\theta_{|A|}.$$
			
			On the other hand, if $|A|>n$, (\ref{eq: from jump to ngbh}) gives
			$$(1+\e_{|A|})\|P_A(y)\|_{(s),A} \stackrel{(\ref{jump})}{\leq} \frac{1+\e_{n+1}} {1+\e_n}+\theta_{n+1} \stackrel{(\ref{seq2})}{\leq} 1-\theta_{n+1}\leq 1-\theta_{|A|}.$$
		\end{proof}
		
		Next, we shall glue together in a smooth way the norms $\{\n_{(s),A}\}_{|A|<\omega}$. For every $n<\omega$, let $\rho_n\colon \R\to [0,\infty)$ be a $C^{\infty}$-smooth, even, and convex function such that $\rho_n\equiv0$ on $[0,1- \theta_n^2]$ and $\rho_n(1)=1$. Then, for every $n<\omega$, $\rho_n$ is strictly monotonically increasing in $[1-\theta_n^2,\infty)$. Define $\Psi\colon Y\to[0,\infty]$ by
		\begin{equation*}
		\Psi(x):=\sum_{|A|<\omega} \rho_{|A|}\left((1+\e_{|A|})\cdot (1+\theta_{|A|})\cdot\|P_A(x)\|_{(s),A}\right) \ \ \ (x \in Y).
		\end{equation*}
		
		We first prove that $\Psi$ is (real-valued and) $C^\infty$-smooth on the open set $\mathcal{O}$ defined by
		\begin{equation*}
		\mathcal{O} := \bigcup \left\{\mathcal{O}_x\colon x\in Y, \|x\|_f\leq1  \right\}.	
		\end{equation*}
		Indeed, we actually show that $\Psi$ is locally expressed by a finite sum on $\mathcal{O}$ (whence the claim follows, since every summand is plainly $C^\infty$-smooth). Pick $x \in Y$ with $\|x\|_f\leq1$ and let $y\in\mathcal{O}_x$. Then, for every finite set $A$ with $A \nsubseteq\supp(x)$, (\ref{jump-ngh}) yields us
		\begin{equation*}
		(1 + \e_{|A|})(1 + \theta_{|A|})\|P_A(x)\|_{(s),A} \leq (1 + \theta_{|A|})(1 - \theta_{|A|}) = 1 - \theta_{|A|}^2,	
		\end{equation*}
		which implies that $\rho_{|A|}\left((1+\e_{|A|})(1+\theta_{|A|}) \|P_A(x)\|_{(s),A}\right)=0$. This shows that, on the open set $\mathcal{O}_x$, only the finitely many terms with $A\subseteq \supp x$ give a non-zero contribution to $\Psi$.\smallskip
		
		Moreover, we note that
		\begin{equation} \label{approx}
		\left\{x\in Y\colon\|x\|_f\leq\frac{1-\theta_1}{1+\theta_1} \right\} \subseteq\left\{\Psi\leq 1\right\} \subseteq\left\{x\in Y\colon\|x\|_f\leq 1\right\}\subseteq\mathcal{O}.	
		\end{equation}
		Indeed, the last inclusion is immediate by the definition of the set $\mathcal{O}$. If we take $x \in Y$ with $\Psi(x) \leq 1$, then we have that $\rho_{|A|}\left((1+\e_{|A|})(1+\theta_{|A|}) \|P_A(x)\|_{(s),A}\right)\leq1$ for every finite set $A$. By the properties of the functions $\rho_n$, we then get $(1+\e_{|A|})(1+\theta_{|A|}) \|P_A(x)\|_{(s),A}\leq1$. Therefore, by (\ref{smooth-norm}),
		\begin{equation*}
		1 \geq (1 + \e_{|A|})(1 + \theta_{|A|})\|P_A(x)\|_{(s),A} \geq (1 + \e_{|A|}) \|P_A(x)\|,	
		\end{equation*}
		which implies that $\|x\|_f \leq 1$. This shows the inclusion $\{ \Psi \leq 1 \} \subseteq \{\n_f \leq 1 \}$. Finally, if $x \in Y$ satisfies $\|x\|_f\leq\frac{1-\theta_1}{1+\theta_1}$, then, for every finite set $A$,
		
		\begin{eqnarray*}
			(1+\e_{|A|})(1+\theta_{|A|})\|P_A(x)\|_{(s),A} &\leq& (1+\e_{|A|})(1+\theta_{|A|})\|P_A(x)\| \\
			&\leq& \frac{1-\theta_1}{1+\theta_1}(1+\theta_{|A|}) \leq 1-\theta_1 \leq 1-\theta_{|A|}^2.
		\end{eqnarray*}
		So, we actually showed that if $x\in Y$ satisfies $\|x\|_f \leq \frac{1-\theta_1}{1+\theta_1}$, then $\Psi(x)=0$.
		
		\vspace{0.2cm}
		
		For the last part of the proof, we would like to adjust all the information we have so far in order to apply Lemma \ref{minkowlemma}. Consider the convex set $D:=\{\Psi<1\}$; by (\ref{approx}), we know that $D\subseteq\mathcal{O}$, whence $D$ is an open set in $Y$. Moreover, $\Psi$ is even, convex, and $C^\infty$-smooth on $D$. (Notice that we cannot apply Lemma \ref{minkowlemma} directly to the set $\mathcal{O}$, since it is not \emph{a priori} clear that $\mathcal{O}$ is convex.) Also, the level set $\{\Psi\leq1-\theta_1\}$ is a closed subset of $Y$, by the lower semi-continuity of $\Psi$. Therefore, we can apply Lemma \ref{minkowlemma}, which assures us that the Minkowski functional $\nn\cdot$ of $\{\Psi\leq1-\theta_1\}$ is an equivalent $C^{\infty}$-smooth norm on $Y$.  
		
		Finally, by using the previous observations and (\ref{approx}), we get that
		\begin{equation*}
		\n_f\leq \nn\cdot\leq \frac{1+\theta_1}{1-\theta_1}\n_f;
		\end{equation*}
		combining with (\ref{f-normal-approx}), we conclude that $\nn\cdot$ is a $C^{\infty}$-smooth norm on $Y$ which approximates $\n$, as desired.	
	\end{proof}		
	
	As a particular case of Theorem \ref{unconditional}, we have the following corollary.
	\begin{corollary} The linear span of the canonical basis of $\ell_p(\Gamma)$ ($1\leq p<\infty$) admits a $C^{\infty}$-smooth norm.
	\end{corollary} 
	
	In the case $p=1$, we can argue as in the proof of Corollary \ref{Cor: analytic in dense of l1c} and obtain the following stronger result.	
	\begin{corollary} Every dense subspace of $\ell_1(\Gamma)$ contains a further dense subspace which admits a $C^{\infty}$-smooth norm.
	\end{corollary}

	\section{Proof of Theorem C}
	In this section, we shall show that our techniques in Section \ref{Sec: Analytic} admit no extension to larger Banach spaces. More precisely, we shall prove that Theorem \ref{equiv-analytic-norm-main} admits no extension to the spaces $\ell_\infty^c(\Gamma)$, $\Gamma$ uncountable (see Corollary \ref{Cor: no Gateaux norm}). The main result of the section is the forthcoming strengthening of Theorem C.
	
	\begin{theorem}\label{Th: c00 in renorming} Let $\Gamma$ be a cardinal number with ${\rm cf}\,\Gamma\geq \omega_1$. Then, every renorming of $\ell_{\infty}^{c,F}(\Gamma)$ contains an isometric copy of $\ell_\infty^{c,F}(\Gamma)$.
	\end{theorem}
	
	As a particular case of the result, every renorming of $\ell_{\infty}^{c,F}(S)$ contains an isometric copy of $c_{00}(\omega_1)$, whenever $S$ is an uncountable set. Therefore, we arrive at the following corollary.
	\begin{corollary}\label{Cor: no Gateaux norm} $\ell_{\infty}^{c,F}(S)$ does not admit a G\^ateaux differentiable norm, whenever $S$ is an uncountable set.
	\end{corollary}
	
	In the proof of Theorem \ref{Th: c00 in renorming} we shall need the following standard lemma, a straightforward consequence of a representation theorem for $\ell_\infty^{c}(S)^*$; we also give a simple, self-contained proof, for the sake of completeness.
	\begin{lemma}\label{lemma1} For every $\p\in(\ell_\infty^c(S))^*$, there is a countable set $A\subseteq S$ such that $\langle\p,x\rangle=0$ for every $x\in\ell_\infty^c(S)$ such that $\supp x\cap A=\emptyset$.
	\end{lemma}
	
	\begin{proof} If $S$ is countable, the result is trivially true, as we can take $A=S$; we then assume that $S$ is uncountable. Arguing by contradiction, assume that, for every countable subset $A$ of $S$, it is possible to find a unit vector $x\in\ell_\infty^c(S)$ with $\langle\p, x\rangle>0$ and $\supp x\cap A=\emptyset$. By an obvious transfinite induction argument, we then obtain a disjointly supported long sequence $(x_\alpha)_{\alpha< \omega_1}$ of unit vectors such that $\langle\p, x_\alpha\rangle>0$, for each $\alpha<\omega_1$. Therefore, up to passing to an uncountable subset and relabeling, we can also assume that $\langle\p, x_\alpha\rangle>\delta$, for some $\delta>0$ and every $\alpha<\omega_1$. This is, however, impossible: indeed, for every $N\in\N$,
		$$1=\left\|\sum_{j=0}^N x_j\right\|_\infty\geq \left\langle\p,\sum_{j=0}^N x_j\right\rangle \geq\delta N,$$
		a contradiction.
	\end{proof}
	
	\begin{corollary}\label{Cor: forbidden set} Let $\n$ be any equivalent norm on $\ell_\infty^c(S)$. Then, for every $w\in\ell_\infty^c(S)$, there exists a countable subset $A$ of $S$ such that $\supp w\subseteq A$ and 
		$$\|w+u\|\geq\|w\|$$
		for every $u\in\ell_\infty^c(S)$ with $\supp u\cap A=\emptyset$.
	\end{corollary}
	
	\begin{proof} Indeed, let $w\in\ell_\infty^c(S)$. Then, by the Hahn-Banach theorem, there is $\p\in(\ell_\infty^c(S))^*$ with $\|\p\|=1$ and such that $\langle\p,w\rangle=\|w\|$. Consider the countable set $A\subseteq S$ obtained applying Lemma \ref{lemma1} to $\p$. Then, for every $u\in\ell_\infty^c(S)$ such that $\supp u\cap A=\emptyset$,
		$$\|w+u\|\geq\langle\p,w+u\rangle=\langle\p,w\rangle=\|w\|.$$
		Finally, since $\supp w$ is countable, we can also assume $\supp w\subseteq A$, and we are done.
	\end{proof}
	
	\begin{proof}[Proof of Theorem \ref{Th: c00 in renorming}] Let $\n$ be an equivalent norm in $\ell_\infty ^{c,F}(\Gamma)$ and consider the set
		$$\mathcal{U}:=\{x\in\ell_\infty^c(\Gamma)\colon x(\gamma)\in\{0,\pm1\} \ \mbox{for each} \ \gamma\in\Gamma\}.$$
		Evidently, $\mathcal{U}\subseteq\ell_\infty^{c,F}(\Gamma)$; moreover, although $\mathcal{U}$ is not a linear subspace, $u\pm v\in\mathcal{U}$ whenever $u,v\in\mathcal{U}$ are disjointly supported.
		\begin{claim}\label{Claim: constant norm} For every subset $S$ of $\Gamma$ with $|S|=\Gamma$, there exist a vector $V\in\mathcal{U}$ and a countable subset $F$ of $S$ with ${\rm supp}\,V\subseteq F$ such that
			$$\|V+u\|=\|V\|,$$
			for every $u\in\mathcal{U}$ with ${\rm supp}\,u\subseteq S\setminus F$.
		\end{claim}
		
		\begin{proof}[Proof of Claim \ref{Claim: constant norm}] Fix a sequence $(\e_j)_{j=0}^\infty$ of positive scalars with $\e_j\searrow0$; also, select $v_0\in\mathcal{U}$. By Corollary \ref{Cor: forbidden set}, there exists a countable set $A_0\subseteq S$ with $\supp v_0\subseteq A_0$ such that $\|v_0+u\|\geq\|v_0\|$ for every $u\in\mathcal{U}$ with $\supp u\subseteq S\setminus A_0$. We then set
			$$\alpha_0:=\sup\left\{\|v_0+u\|\colon u\in\mathcal{U}, \supp u\subseteq S\setminus A_0\right\};$$
			we can now pick a vector $v_1\in\mathcal{U}$ with $\supp v_1\subseteq S\setminus A_0$ and such that $\|v_0+v_1\|\geq\alpha_0-\e_0$. We then repeat the above argument. Corollary \ref{Cor: forbidden set}, applied to the vector $v_0+v_1$, yields us a countable set $A_1\subseteq S$ with $\supp (v_0+v_1)\subseteq A_1$ and such that $\|v_0+v_1+u\|\geq\|v_0+v_1\|$ for every $u\in\mathcal{U}$ with $\supp u\subseteq S\setminus A_1$. Of course, we can assume without loss of generality that $A_0\subseteq A_1$. Next, we define
			$$\alpha_1:=\sup\left\{\|v_0+v_1+u\|\colon u\in\mathcal{U}, \supp u\subseteq S\setminus A_1\right\}.$$
			
			If we continue by induction, we obtain a sequence of vectors $(v_j)_{j=0}^\infty\subseteq\mathcal{U}$ and an increasing sequence of countable subsets $(A_j)_{j=0}^\infty$ of $S$ such that, upon setting
			$$\alpha_n:=\sup\left\{\|v_0+\dots+v_n+u\|\colon u\in\mathcal{U}, \supp u\subseteq S\setminus A_n\right\},$$
			we have:
			\begin{romanenumerate}
				\item $\supp v_n\subseteq A_n$;
				\item $\supp v_{n+1}\cap A_n=\emptyset$;
				\item $\|v_0+\dots+v_n+ v_{n+1}\|\geq \alpha_n-\e_n$;
				\item $\|v_0+\dots+v_n+ u\|\geq \|v_0+\dots+v_n\|$, for every $u\in\mathcal{U}$ with $\supp u\subseteq S\setminus A_n$.
			\end{romanenumerate}
			Indeed, once we have found $v_0,\dots,v_n$ and $A_0,\dots,A_n$ with the above properties, by definition of $\alpha_n$, we can find $v_{n+1}\in\mathcal{U}$ with $\supp v_{n+1}\subseteq S\setminus A_n$ and such that 
			$$\|v_0+\dots+v_n+ v_{n+1}\|\geq \alpha_n-\e_n.$$
			This gives (ii) and (iii). Then, we apply Corollary \ref{Cor: forbidden set} to the vector $v_0+\dots+v_{n+1}$ and we find a countable subset $A_{n+1}$ of $S$ for which (i) and (iv) hold. We can also assume $A_n\subseteq A_{n+1}$, which concludes the induction step.\smallskip
			
			By (i) and (ii), the vectors $v_j$ are disjointly supported; therefore the series $\sum_{j=0}^\infty v_j$ is pointwise convergent and the vector $V:=\sum_{j=0}^\infty v_j$ is a well-defined element of $\mathcal{U}$. Moreover, for every $n\in\N$, the support of $\sum_{j=n+1}^\infty v_j$ is contained in $S\setminus A_n$, whence $\|V\|\geq\|v_0+\dots+v_n\|$, by (iv).
			
			We are now in position to define the forbidden set $F$. By Corollary \ref{Cor: forbidden set}, we can choose a countable set $B$ such that $\supp V\subseteq B$ and $\|V+u\|\geq\|V\|$ for every $u\in\mathcal{U}$ with $\supp u\subseteq S\setminus B$. We then set $F:= \cup_{j=0}^\infty A_j\cup B$.\smallskip
			
			Finally, fix $u\in\mathcal{U}$ with $\supp u\subseteq S\setminus F$. On the one hand, $\|V+u\|\geq\|V\|$, since $\supp u\subseteq S\setminus B$. On the other one, for every $n\in\N$, the vector $u+\sum_{j=n+1}^\infty v_j$ is supported in $S\setminus A_n$, whence the definition of $\alpha_n$ and (iii) imply
			$$\|V+u\|=\left\|v_0+\dots+v_n+ \left(u+\sum_{j=n+1}^\infty v_j\right) \right\|\leq\alpha_n\leq \|v_0+\dots+v_n+v_{n+1}\|+\e_n\leq \|V\|+\e_n.$$
			Letting $n\to\infty$ yields $\|V+u\|=\|V\|$ and concludes the proof of the claim. 
		\end{proof}
		
		Having the claim at our disposal, we can find by transfinite induction vectors $(V_\alpha)_{\alpha<\Gamma}\subseteq\mathcal{U}$ and countable disjoint sets $(F_\alpha)_{\alpha<\Gamma}$ with $\supp V_\alpha\subseteq F_\alpha$ such that
		$$\|V_\alpha+u\|=\|V_\alpha\| \;\; \mbox{for every} \ u\in\mathcal{U} \ \mbox{with} \ \supp u\cap \left(\bigcup_{\beta\leq \alpha}F_\beta \right)=\emptyset.$$
		Indeed, having found $(V_\alpha)_{\alpha<\gamma}$ and $(F_\alpha)_{\alpha<\gamma}$, for some $\gamma<\Gamma$, we apply Claim \ref{Claim: constant norm} to the set $\Gamma\setminus \left(\cup_{\alpha<\gamma}F_\alpha \right)$ and obtain the desired vector $V_\gamma$ and countable set $F_\gamma$. (In the case $\gamma=0$, we just apply the claim to $S=\Gamma$.)\smallskip
		
		If we select any increasing sequence $(\alpha_j)_{j=0}^\infty$ of ordinals, with $\alpha_j<\Gamma$ for each $j$, and signs $(\e_j)_{j=0}^\infty \subseteq\{\pm1\}$, it is clear that the vector $\sum_{j=1}^\infty\e_j V_{\alpha_j}$ belongs to $\mathcal{U}$ and its support is disjoint from $\cup_{\beta\leq\alpha_0}F_\beta$ (once more, the convergence of the above series is intended in the pointwise sense). Therefore,
		\begin{equation}\label{eq: norm equal}
		\left\|\sum_{j=0}^\infty\e_j V_{\alpha_j}\right\|=\| V_{\alpha_0}\|.
		\end{equation}
		
		In particular, $\|V_\alpha\pm V_\beta\|=\|V_\alpha\|$ when $\alpha<\beta<\Gamma$, whence $\|V_\beta\|\leq \frac{1}{2}\|V_\alpha +V_\beta\| +\frac{1}{2}\|V_\alpha- V_\beta\| \\=\|V_\alpha\|$. By the uncountable cofinality of $\Gamma$, the non-increasing function $\alpha\mapsto \|V_\alpha\|$ ($\alpha<\Gamma$) is therefore eventually constant; consequently, up to discarding some initial terms and relabeling, we can assume that there exists $c\in(0,\infty)$ with $\|V_\alpha\|=c$, for each $\alpha<\Gamma$.
		
		From (\ref{eq: norm equal}) we can now infer that, for every increasing sequence $(\alpha_j)_{j=0}^\infty$ of ordinals smaller than $\Gamma$ and signs $(\e_j)_{j=0}^\infty \subseteq\{\pm1\}$, we have
		$$\left\|\sum_{j=0}^\infty\e_j V_{\alpha_j}\right\|=c.$$
		A standard convexity argument then assures us that, for every finitely valued sequence $(c_j)_{j=0}^\infty$ of scalars, one has
		$$\left\|\sum_{j=0}^\infty c_j V_{\alpha_j}\right\|= \max_{j=0,\dots,\infty}|c_j|.$$
		Therefore, the rule
		$$\ell_\infty^{c,F}(\Gamma) \ni(c_\alpha)_{\alpha<\Gamma} \mapsto\sum_{\alpha<\Gamma} c_\alpha V_\alpha$$
		defines an isometric embedding of $\ell_\infty^{c,F}(\Gamma)$ into $(\ell_\infty^{c,F}(\Gamma),\n)$, which concludes the proof.
	\end{proof}
	
\medskip{}
\textbf{Acknowledgements.} The authors wish to express their gratitude to the anonymous referee for carefully reading their manuscript and for the detailed and helpful report.


\end{document}